\documentclass[12pt]{amsart}

\usepackage{amsmath}
\usepackage{amssymb}
\usepackage{ascmac}
\usepackage{amsthm}

\makeatletter

\@addtoreset{equation}{section}
\makeatother 

\theoremstyle{plain}
\newtheorem{thm}{Theorem}[section]
\newtheorem*{thm*}{Theorem}

\newtheorem{lem}[thm]{Lemma}
\newtheorem{cor}[thm]{Corollary}

\theoremstyle{definition}

\theoremstyle{remark}
\newtheorem{rem}{Remark}[section]
\DeclareMathOperator*{\esssup}{ess\,sup}
\newcommand{\vol}{\operatorname{vol}}

\newcommand{\diam}{\operatorname{Diam}}

\usepackage{latexsym}

\title[Some universal inequalities]{Some universal inequalities of eigenvalues and upper bounds for the $L_{\infty}$ norm of eigenfunctions of the Laplacian}

\author{Kei Funano}
\address{Division of Mathematics \& Research Center for Pure and Applied Mathematics, Graduate School of Information Sciences, Tohoku University, 6-3-09 Aramaki-Aza-Aoba, Aoba-ku, Sendai 980-8579, Japan}
\email{kfunano@tohoku.ac.jp}

\date{\today}

\begin{document}
\maketitle

\begin{abstract}
    In this short survey, we derive some weyl-type universal inequalities of eigenvalues of the Laplacian on a closed Riemannian manifold of nonnegative Ricci curvature. We also give upper bounds for the $L_{\infty}$ norm of eigenfunctions of the Laplacian in the same setting. A detailed proof of these results did not seem to appear in the literature but the results follow from a simple combination of Milman's work and Cheng-Li's work.
\end{abstract}

\section{Introduction}

For a closed Riemannian manifold $M$ let $\lambda_k(M)$ be the $k$-th nonzero eigenvalue of the Laplacian.

One of the goal of this survey is to give a precise proof of the following theorem: For $a,b\in \mathbb{R}$, $a\lesssim b$ stands for $a\leq Cb$ for some universal (and numerical) constant $C>0$.
\begin{thm}\label{weyltype}Let $M$ be an $n$-dimensional closed Riemannian manifold of nonnegative Ricci curvature. Then for any $k\geq 1$ we have
\begin{align}\label{Milmanineq}
\lambda_k(M)\gtrsim k^{\frac{2}{n}}\lambda_1(M).
\end{align}
\end{thm}Note that the order of $k$ in (\ref{Milmanineq}) is sharp by the Weyl asymptotic law. Concerning the above inequality (\ref{Milmanineq}) let us focus on the following two estimates for eigenvalues of the Laplacian on a closed Riemannian manifold of nonnegative Ricci curvature:
\begin{enumerate}
    \item $\lambda_k(M)\geq \frac{c_n k^{\frac{2}{n}}}{(\diam M )^2}$, where $c_n>0$ is a constant depending only on $n$ (Gromov \cite{Gr2,Gr}. Refer also to \cite[Theorem 4.2]{SY} and \cite[Theorems 1.1 and 1.2]{HKP}).
    \item $\lambda_k(M)\lesssim \frac{n^2 k^2}{(\diam M )^2}$ (Cheng \cite[Corollary 2.2]{C}).
\end{enumerate}Combining the above two estimates we can obtain 
\begin{align}\label{muriyari}
    \lambda_k(M)\gtrsim c_n k^{\frac{2}{n}}\lambda_1(M)
\end{align}in the same setting. Since we do not need the dimensional term $c_n$ in (\ref{Milmanineq}), the inequality (\ref{Milmanineq}) is better than the inequality (\ref{muriyari}).

The opposite of (\ref{Milmanineq}) was obtained by  Liu in \cite[Theorem 1.1]{Liu}. Under the same assumption of Theorem \ref{weyltype} Liu showed 
\begin{align}\label{liuineq}
    \lambda_k(M)\lesssim k^2\lambda_1(M),
\end{align}which is sharp with respect to the order of $k$. 

Throughout this survey $\mu$ denotes the uniform probability measure on $M$, i.e., $\mu(\cdot):=\vol (\cdot)/\vol(M)$. For a measurable function $f:M\to \mathbb{R}$ and $p>0$ we put
\begin{align*}
\|f\|_p:=\Big(\int_{M}|f(x)|^p d\mu(x)\Big)^{\frac{1}{p}}
\end{align*}and 
\begin{align*}
    \|f\|_{\infty}:=\esssup_{x\in M}|f(x)|.
\end{align*}

We also give the proof of the following theorem.
\begin{thm}\label{eigenfctconc}
    Let $n\geq 3$ and $M$ be an $n$-dimensional closed Riemannian manifold of nonnegative Ricci curvature. Then for any $k\geq 1$ and any eigenfunction $\varphi_k$ of the Laplacian corresponding to $\lambda_k(M)$ we have
\begin{align}\label{supnorm}
\|\varphi_k\|_{\infty}\lesssim \Big(\frac{c\lambda_k(M)}{\lambda_1(M)}\Big)^{\frac{n}{4}}\|\varphi_k\|_2.
\end{align}for some universal constant $c>0$. In particular we have
\begin{align}\label{inftynorm}
\|\varphi_k\|_{\infty}\lesssim (ck)^{\frac{n}{2}}\|\varphi_k\|_2.
\end{align}
\end{thm}
The inequality (\ref{inftynorm}) follows from (\ref{supnorm}) together with (\ref{liuineq}).

A classical result by H\"{o}lmander \cite{H} (see also \cite{A}, \cite{DG}, \cite{L}, \cite{S}, and \cite
[Corollary 8.6]{Z}) states that
\begin{align}\label{local}
    \|\varphi_k\|_{\infty}\leq C \lambda_k(M)^{\frac{n-1}{4}}\|\varphi_k\|_2
\end{align}for any closed Riemannian manifold, where the constant $C=C(M,g)$ depends on $(M,g)$. This inequality is sharp and better than (\ref{supnorm}) when $\lambda_k(M)$ is large. In the case where $\lambda_k(M)<<\lambda_1(M)^n$ the inequality (\ref{supnorm}) is better than the inequaity (\ref{local}).

The proof of Theorems \ref{weyltype} and \ref{eigenfctconc} follows from a simple combination of Milman's result \cite{Mi1} and Cheng-Li's result \cite{CL}. In fact under the nonnegativity of Ricci curvature, Milman proved the Sobolev inequality in terms of the first nonzero eigenvalue and Cheng-Li proved lower bounds of eigenvalues of the Laplacian in terms of the Sobolev constant. As a byproduct of the proof of Theorem \ref{weyltype} we prove Theorem \ref{eigenfctconc}.

\begin{rem}\upshape Theorems \ref{weyltype} and \ref{eigenfctconc} hold for a convex domain in a complete Riemannian manifold of nonnegative Ricci curvature. In that case we impose the Neumann boundary condition.
\end{rem}

\section{Proof}
Let $M$ be a closed Riemannian manifold and $\mu$ be the uniform probability measure. For the isoperimetric profile $I=I_{(M,\mu)}:[0,1]\to \mathbb{R}$ (Refer to \cite{Mi1} for the definition) we put $\tilde{I}(t):=\min \{ I(t), I(1-t)\}$, $t\in [0,1/2]$. 
\begin{thm}[{E.~Milman, \cite[Remark 2.11]{Mi1}}]\label{Milmanisop}Let $M$ be an $n$-dimensional closed Riemannian manifold of nonnegative Ricci curvature. Then we have the isoperimetric inequality
\begin{align*}
\tilde{I}(t)\gtrsim \sqrt{\lambda_1(M)}t^{\frac{n-1}{n}} \ \ \forall t\in [0,1/2].
\end{align*}
\end{thm}
In \cite[Remark 2.11]{Mi1} Milman stated that 
\begin{align*}
    \tilde{I}(t)\gtrsim D_{1,\infty}t^{\frac{n-1}{n}} \ \ \forall t\in [0,1/2],
\end{align*}where $D_{1,\infty}$ is the best constant  that satisfies the $(1,\infty)$-Poincaré inequality. Since $D_{1,\infty}$ is equivalent to $\sqrt{\lambda_1(M)}$ by Milman's result (\cite[Theorem 2.4]{Mi1}), we have Theorem \ref{Milmanisop}.

For a measurable function $f:M\to \mathbb{R}$ we define a \emph{median} of $f$ as a number $a$ such that 
\begin{align*}
    \mu(f(x)\geq a)\geq \frac{1}{2} \text{ and } \mu(f(x)\leq a)\geq \frac{1}{2}.
\end{align*}Note that $a$ might not be unique. Let $M_{\mu}f$ denote one of medians of $f$. 
We put 
\begin{align*}
    E_{\mu}f:=\int_M f(x)d\mu
\end{align*}whenever the right-hand side makes sense.

One can transfer the isoperimetric inequality appeared in Theorem \ref{Milmanisop} to a $(1, \frac{n}{n-1})$-Poincar\'e inequality via the following theorem. Let $\mathcal{F}$ be the space of Lipschitz functions on $M$. 
\begin{thm}[{Bobkov–Houdr\'e \cite{BH}, Federer–Fleming \cite{FF}, Maz'ya \cite{Ma}}]\label{sobisop}Let $0<r\leq 1$. Without any assumption on $M$ the $(1/r,1)$-Poincar\'e inequality 
\begin{align*}
D\|f-M_{\mu}f \|_{\frac{1}{r}}\leq \| | \nabla f | \|_{1}\ \ \forall f\in \mathcal{F}
\end{align*}is equivalent to the isoperimetric inequality
\begin{align*}\tilde{I}(t)\geq Dt^r \ \ \forall t\in [0,1/2].
\end{align*}
\end{thm}See \cite[Theorem 1.1]{BH} for the proof of the above theorem.

Theorem \ref{Milmanisop} together with Theorem \ref{sobisop} yields the following.
\begin{cor}\label{milmancor}Let $M$ be an $n$-dimensional closed Riemannian manifold of nonnegative Ricci curvature. Then we have the $(\frac{n}{n-1},1)$-Poincar\'e inequality  
\begin{align*}
    \sqrt{\lambda_1(M)}\|f-M_{\mu}f \|_{\frac{n}{n-1}} \lesssim \| |\nabla f|\|_{1}\ \ \forall f\in \mathcal{F}.
    \end{align*}
\end{cor}
From now on we shall study Poincar\'e inequalities.
\begin{lem}[{\cite[Section 1.1.3]{L}}]\label{1top}Suppose that the $(\frac{n}{n-1},1)$-Poincar\'e inequality 
\begin{align}\label{s1}
    D\|f-M_{\mu}f \|_{\frac{n}{n-1}}\leq \| |\nabla f|\|_{1}\ \ \forall f\in \mathcal{F}
\end{align}holds for $M$. Then for any $p\geq 1$ we have the $(\frac{np}{n-p},p)$-Poincar\'e inequality 
\begin{align}\label{s2}
D\| f-M_{\mu}f     \|_{\frac{np}{n-p}}\leq \frac{(n-1)p}{n-p}\|  |\nabla f| \|_p\ \ \forall f\in \mathcal{F}.
\end{align}
\end{lem}
In \cite[Section 1.1.3]{L} the above lemma is stated for $\mathbb{R}^n$. In there the Poincar\'e inequalities corresponding to (\ref{s1}) and (\ref{s2}) are stated for compactly supported smooth functions and the left-hand sides of (\ref{s1}) and (\ref{s2}) in there are the norms of $f$, not the norms of $f-M_{\mu}f$.

\begin{proof}[Proof of Lemma \ref{1top}]
Assume (\ref{s1}) and take $f\in \mathcal{F}$. We may assume that $M_{\mu}f=0$. Setting $\alpha:= \frac{(n-1)p}{n-p}$ we find $\alpha>1$.
For any $x\in M$ define \begin{align*}
    g(x):=\left\{
\begin{array}{ll}
|f(x)|^{\alpha} & \text{ if }f(x)\geq 0 \\
 -|f(x)|^{\alpha} & \text{ if }f(x)< 0.
\end{array}
\right.
\end{align*}Observe that $g\in \mathcal{F}$ and $M_{\mu}g=0$. Thus (\ref{s1}) yields that
\begin{align*}
    D\|f\|^{\alpha}_{\frac{\alpha n }{n-1}}=D\|g\|_{\frac{n}{n-1}}\leq \alpha \int |f(x)|^{\alpha-1}|\nabla f(x)|d\mu(x).
\end{align*}Let $p'$ be the conjugate of $p$, i.e., $p^{-1}+p'^{-1}=1$. Using the H\"older inequality we obtain
\begin{align*}
    D\|f\|_{\frac{np}{n-p}}^{\frac{(n-1)p}{n-p}}=D\|f\|^{\alpha}_{\frac{\alpha n }{n-1}} \leq  \alpha \||f|^{\alpha-1}\|_{p'}\||\nabla f|\|_p=\alpha \|f\|_{\frac{np}{n-p}}^{\frac{n(p-1)}{n-p}}\||\nabla f|\|_p,
    \end{align*}which gives (\ref{s2}). This completes the proof.
\end{proof}
To connect the $L_p$ norm of $f-E_{\mu}f$ with $f-M_{\mu}f$ we use the following.
\begin{lem}[{\cite[Lemma 2.1]{Mi1}}]\label{medexp}For any $p\geq 1$ and $f\in \mathcal{F}$ we have
\begin{align*}
    \frac{1}{2}\|f-E_{\mu}f\|_p\leq \| f-M_{\mu}f\|_p \leq 3\| f-E_{\mu}f\|_p.
\end{align*}
\end{lem}
Lemma \ref{1top} together with Lemma \ref{medexp} gives the following.
\begin{lem}[{\cite[Section 1.1.3]{L}}]\label{Coste}Suppose that the $(\frac{n}{n-1},1)$-Poincar\'e inequality 
\begin{align*}
    D\|f-E_{\mu}f \|_{\frac{n}{n-1}}\leq \| |\nabla f|\|_{1}\ \ \forall f\in \mathcal{F}
\end{align*}holds for $M$. Then for any $p\geq 1$ we have the $(\frac{np}{n-p},p)$-Poincar\'e inequality 
\begin{align*}
D\| f-E_{\mu}f     \|_{\frac{np}{n-p}}\lesssim \frac{(n-1)p}{n-p}\|  |\nabla f| \|_p\ \ \forall f\in \mathcal{F}.
\end{align*}
\end{lem}

Combining Corollary \ref{milmancor} with Lemmas \ref{medexp} and \ref{Coste} we obtain the following.

\begin{cor}\label{poin}Let $M$ be an $n$-dimensional closed Riemannian manifold of nonnegative Ricci curvature. Then for any $p\geq 1$ we have the $(\frac{np}{n-p},p)$-Poincar\'e inequality  
\begin{align*}
    \sqrt{\lambda_1(M)}\|f-E_{\mu}f \|_{\frac{np}{n-p}} \lesssim \frac{(n-1)p}{n-p}\| |\nabla f|\|_{p}\ \ \forall f\in \mathcal{F}.
\end{align*}
\end{cor}

The H\"older inequality implies the following lemma.
\begin{lem}\label{holder}Suppose that we have the $(\frac{2n}{n-2},2)$-Poincar\'e inequality
\begin{align*}
D\| f-E_{\mu}f\|_{\frac{2n}{n-2}}\leq \| |\nabla f|\|_2 \ \ \forall f\in \mathcal{F}.
\end{align*}Then we have
\begin{align*}
D \|f-E_{\mu}f\|_2^{\frac{2+n}{n}}\| f-E_{\mu}f\|_{1}^{-\frac{2}{n}}\leq \| |\nabla f|\|_2\ \ \forall f\in \mathcal{F}.
\end{align*}
\end{lem}

\begin{thm}[{Cheng-Li, \cite[Theorem 1]{CL}}]\label{chenglithm}Let $M$ be a closed Riemannian manifold. Assume that 
\begin{align}\label{assump}
D \|f-E_{\mu}f\|_2^{\frac{2+n}{n}}\| f-E_{\mu}f\|_{1}^{-\frac{2}{n}}\leq \| |\nabla f|\|_2 \ \ \forall f\in \mathcal{F}
\end{align}holds. Then we have
\begin{align*}
\lambda_k(M)\gtrsim k^{\frac{2}{n}}D^2
\end{align*}for any $k$.
\end{thm}

\begin{proof}[Proof of Theorem \ref{weyltype}]The case where $n=2$ follows from (\ref{muriyari}) and hence we assume $n\geq 3$. 

Since $n\geq 3$ we have 
\begin{align*}
    \sqrt{\lambda_1(M)}\|f-E_{\mu}f \|_{\frac{2n}{n-2}} \lesssim \| |\nabla f|\|_{2}\ \ \forall f\in \mathcal{F}.
\end{align*}by Corollary \ref{poin}. From Lemma \ref{holder} we have
    \begin{align*}
\sqrt{\lambda_1(M)} \|f-E_{\mu}f\|_2^{\frac{2+n}{n}}\| f-E_{\mu}f\|_{1}^{-\frac{2}{n}}\lesssim \| |\nabla f|\|_2\ \ \forall f\in \mathcal{F}.
\end{align*}Theorem \ref{chenglithm} thereby implies the theorem. This completes the proof.
\end{proof}
For the completeness of this paper we give the proof of Theorem \ref{chenglithm} following \cite{CL}.
\begin{proof}[Proof of Theorem \ref{chenglithm}]Let $H(x,y,t)$ denotes the fundamental solution of the heat equation, i.e., 
\begin{align*}
    \frac{\partial H}{\partial t}(x,y,t)=\Delta_x H(x,y,t)
\end{align*}and 
\begin{align*}
    \int_M H(x,y,t)f(y)d\mu(y)\to f(x) \text{ as }t\to 0 
\end{align*}for any $f\in C^{\infty}(M)$ (\cite{Gr2,Gr}). We define 
$G:M\times M\times [0,\infty) \to \mathbb{R}$ as $G(x,y,t):=H(x,y,t)-1$. Denote by $\{\phi_i\}_{i=0}^{\infty}$ be an orthnormal basis of eigenfunctions of the Laplacian with respect to $\mu$, where $\phi_i$ corresponds to the eigenvalue $\lambda_k(M)$. Since $\phi_0=1$ we have an expansion 
\begin{align*}
    G(x,y,t)=\sum_{i=1}^{\infty}e^{-\lambda_i(M) t }\phi_i(x)\phi_i(y).
\end{align*}Note also that the semi-group property
\begin{align*}
G(x,y,t+s)=\int_M G(x,z,s)G(z,y,t)d\mu(z) \ \ \forall x,y\in M \ \forall s,t\geq 0
\end{align*}holds. We thereby have
\begin{align*}
G'(x,x,t)=\ &\int_M G'(x,y,t/2)G(x,y,t/2)d\mu(y)\\
         =\ &\int_M \Delta_y G(x,y,t/2)G(x,y,t/2)d\mu(y).
\end{align*}
Since
\begin{align*}
\int_M|G(x,y,t)|d\mu(y)\leq \int_M |H(x,y,t)|d\mu(y)+1=2,
\end{align*}the assumption (\ref{assump}) and the semi-group property give 
\begin{align*}
-G'(x,x,t)=\ &\int_M |\nabla_y G(x,y,t/2) |^2d\mu(y)\\
          \geq \ &2^{-\frac{4}{n}}D^2\Big( \int_M G(x,y,t/2)^2 d\mu(y)\Big)^{\frac{2+n}{n}}\\
          = \ & 2^{-\frac{4}{n}}D^2 G(x,x,t)^{\frac{2+n}{n}}.
\end{align*}Since $G(x,x,t)=H(x,x,t)-1 \to \infty$ as $t\to 0$, we obtain
\begin{align*}
\frac{n}{2}G(x,x,t)^{-\frac{2}{n}}\geq 2^{-\frac{4}{n}}D^2t
\end{align*}It gives 
\begin{align}\label{infoconc}
    G(x,x,t)\leq 4\Big(\frac{2D^2}{n}\Big)^{-\frac{n}{2}}t^{-\frac{n}{2}}.
\end{align}Integrating both-sides yields
\begin{align*}
    \sum_{i=1}^{\infty}e^{-\lambda_i(M) t}\leq 4\Big(\frac{2D^2}{n}\Big)^{-\frac{n}{2}}t^{-\frac{n}{2}}.
\end{align*}
Putting $t:=n/(2\lambda_k(M))$ we therefore obtain
\begin{align*}
ke^{-\frac{n}{2}}\leq \sum_{i=1}^k e^{-\frac{\lambda_i(M)}{2\lambda_k(M)}n}\leq \sum_{i=1}^{\infty} e^{-\frac{\lambda_i(M)}{2\lambda_k(M)}n}\leq 4\Big(\frac{2\lambda_k(M)}{D^2}\Big)^{\frac{n}{2}},
\end{align*}which implies the conclusion of the theorem. This completes the proof.
\end{proof}

In \cite{CL} instead of putting $t=n/(2\lambda_k(M))$ in the above proof, Cheng-Li put $t=1/(2\lambda_k(M))$ and the conclusion in there becomes $\lambda_k(M)\gtrsim c_n k^{\frac{2}{n}}D^2$, where $c_n>0$ is a constant depending only on $n$. In \cite[Chapter III §5]{SY} it is putted $t=n/(4\lambda_k(M))$ but the assumption there is different. The assumption there seems the $(\frac{2n}{n-2},2)$-Poincar\'e inequality
\begin{align*}
D\| f\|_{\frac{2n}{n-2}}\leq \| |\nabla f|\|_2 \ \ \forall f\in \mathcal{F}
\end{align*}to derive the lower bound of $\lambda_k(M)$.

We also give the proof of Theorem \ref{eigenfctconc} following \cite[Corollary 1]{CL}.
\begin{proof}[Proof of Theorem \ref{eigenfctconc}]
   Corollary \ref{poin}, Lemma \ref{holder} and (\ref{infoconc}) yield
   \begin{align*}
      e^{-\lambda_k(M)t} \phi_k(x)^2\leq 4\Big(\frac{c\lambda_1(M)}{n}\Big)^{-\frac{n}{2}}t^{-\frac{n}{2}},
   \end{align*}where $c>0$ is a universal constant. Thus substituting $t:=n/(2\lambda_k(M))$ we obtain (\ref{supnorm}). This completes the proof.
\end{proof}










\begin{thebibliography}{9999}
\bibitem{A}G.~Avakumovic, \textit{\"{U}ber die Eigenfunktionen auf geschlossenen Riemannschen Mannigfaltigkeiten}, Math. Z. 65 (1956), p. 327--344.
\bibitem{BH}S.~G.~Bobkov and C.~Houdr\'e, \textit{Some connections between isoperimetric and Sobolev-type inequalities}.
Mem. Am. Math. Soc. 129(616) (1997). 
\bibitem{C}S.~Y.~Cheng, 
\textit{Eigenvalue comparison theorems and its geometric applications}. Math. Z.143(1975), no.3, 289--297.
\bibitem{CL}S.~Y.~Cheng and P.~Li, \textit{Heat kernel estimates and lower bound of eigenvalues}. Comment.Math. Helv. 56 (1981), no. 3, 327--338.
\bibitem{DG}J.~J.~Duistermaat and V.~W.~Guillemin, \textit{The spectrum of positive elliptic operators and periodic bicharacteristics}. Invent Math 29, 39--79 (1975).
\bibitem{FF}H.~Federer and W.~H.~Fleming, \textit{Normal and integral currents}. Ann. Math. (2) 72, 458--520 (1960).

 \bibitem{Gr2}A.~Grigor'yan, \textit{Heat kernel and analysis on manifolds}. AMS/IP Studies in Advanced Mathematics, 47. American Mathematical Society, Providence, RI; International Press, Boston, MA, 2009. 
 \bibitem{Gr}A.~Grigor'yan, \textit{Heat kernels on weighted manifolds and applications}, The ubiquitous heat kernel, 93–191.
Contemp. Math., 398
American Mathematical Society, Providence, RI, 2006.




 
\bibitem{Gr2}M.~Gromov, \textit{Metric structures for Riemannian and non-Riemannian spaces}. Based on the 1981 French original With appendices by M. Katz, P. Pansu, and S. Semmes. Translated from the French by S. M. Bates. Progress in Mathematics, 152. Birkhäuser Boston, Boston, MA, 1999. 
\bibitem{Gr}M.~Gromov, \textit{Paul Levy's isoperimetric inequality}, Publications IHES, 1980.
\bibitem{HKP}A.~Hassannezhad, G.~Kokarev, and I.~Polterovich, \textit{Eigenvalue inequalities on Riemannian manifolds with a lower Ricci curvature bound}. J. Spectr. Theory6(2016), no.4, 807--835.
\bibitem{H}L.~H\"{o}rmander, \textit{The spectral function of an elliptic operator}, Acta Math. 121 (1968), 193--218.
\bibitem{L}B.~Levitan, \textit{On the asymptotic behavior of the spectral function of a self-adjoint differential
equation of second order}. Isv. Akad. Nauk SSSR Ser. Mat. 16 (1952), p.325--352.
\bibitem{Liu}S.~Liu, \textit{An optimal dimension-free upper bound
         for eigenvalue ratios}, preprint arXiv:1405.2213v3.
\bibitem{Ma}V.~G.~Maz'ja, \textit{Sobolev Spaces}. Springer Series in Soviet Mathematics. Springer, Berlin (1985).
\bibitem{Mi1}E.~Milman, \textit{On the role of convexity in isoperimetry, spectral gap and concentration}, Invent. Math. 177 (2009), no. 1, 1--43.
\bibitem{L}L.~Saloff-Coste, \textit{Aspects of Sobolev-type inequalities}. London Math. Soc. Lecture Note Ser., 289
Cambridge University Press, Cambridge, 2002.
\bibitem{SY}R.~Schoen and S.-T.~Yau, \textit{Lectures on differential geometry}. Lecture notes prepared by Wei Yue
Ding, Kung Ching Chang [Gong Qing Zhang], Jia Qing Zhong and Yi Chao Xu. Translated from the
Chinese by Ding and S. Y. Cheng. Preface translated from the Chinese by Kaising Tso. Conference
Proceedings and Lecture Notes in Geometry and Topology, I. International Press, Cambridge, MA, 1994.
\bibitem{S}C.~D.~Sogge, \textit{Concerning the Lp norm of spectral clusters for second-order elliptic operators on compact manifolds}. J. Funct. Anal.77(1988), no.1, 123--138.
\bibitem{Z}S.~Zelditch, \textit{Eigenfunctions of the Laplacian on a Riemannian manifold}.
CBMS Reg. Conf. Ser. Math., 125
Published for the Conference Board of the Mathematical Sciences, Washington, DC; by the American Mathematical Society, Providence, RI, 2017.
\end{thebibliography}
\end{document}